\documentclass[a4paper,12pt]{article}

\usepackage{amssymb}
\usepackage{epsfig}
\usepackage{amsfonts}
\usepackage{amsmath}
\usepackage{euscript}
\usepackage{amscd}
\usepackage{amsthm}
\usepackage{theoremref}
\usepackage{enumerate}
\usepackage{array}
\usepackage{mathrsfs}
\usepackage{tikz-cd}

\DeclareMathAlphabet{\mathpzc}{OT1}{pzc}{m}{it}
\DeclareMathAlphabet{\mathpzc}{OT1}{pzc}{m}{it}

\newtheorem{thm}{Theorem}[section]
\newtheorem{lem}[thm]{Lemma}
\newtheorem{prop}[thm]{Proposition} 
\newtheorem{cor}[thm]{Corollary}

\newtheorem{rem}[thm]{Remark}
\newtheorem{ex}[thm]{Example}

\newtheorem{question}[thm]{Question}

\newtheorem{Defn}[thm]{Definition}

\title{ 
	On Generalised Danielewski Surfaces over fields of arbitrary characteristic
}
\author{
	Debojyoti Saha\\
	{\small{\it  Theoretical Statistics and Mathematics  Unit, Indian Statistical Institute,}}\\ 
	{\small{\it 203 B.T.Road, Kolkata-700108, India}}\\
	{\small{\it e-mail : debojyotisaha19@gmail.com}}}

\begin{document}
	\date{}
	\maketitle
	
	\abstract
	 In this paper we study exponential maps ($\mathbb{G}_a$-actions) on the  family of affine two dimensional surfaces of the form $f(x)y=\phi(x,z)$ over arbitrary fields, describe the Makar-Limanov invariant and Derksen invariant of these surfaces, give a complete characterization of isomorphisms between such surfaces and  display a subfamily which provides counterexamples to the cancellation problem. 
	
	\smallskip
	\noindent
	{\small {{\bf Keywords}. }}
	\smallskip
	Cancellation Problem, Exponential Map, Makar-Limanov Invariant, Derksen Invariant
	\smallskip
	\newline
	{\small{{\bf 2020 MSC}. Primary: 14R10, 13A50    ; 
			Secondary: 14R05, 13A02, 13B25.

			\section{Introduction}

			\indent
			
			Throughout the paper $R$ will denote a commutative ring with identity, $R^{[n]}$ a polynomial ring over $R$ in $n$ variables and $K$ will denote a field. For an affine $K$-domain $R$, ${\rm trdeg}_K(R)$ will denote the transcendental degree of $R$ over $K$. For convenience, we will call the coordinate ring of an affine surface to be a ``surface".
			\newline 
			In affine algebraic geometry, the Cancellation Problem asks:
            \begin{question}\label{q1}\it{If $R_1$, $R_2$ are two affine $K$ domains such that $R_1^{[1]}\cong_{K}R_2^{[1]}$ then is $R_1\cong_{K}R_2$}?
            \end{question}
            In general the answer to the above question is known to be negative. Complex surfaces, now known as Danielewski surfaces, defined by a polynomial of the form $X^nY-\varphi(Z)$, where $n\in \mathbb{N}$ and $\varphi(Z)\in \mathbb{C}[Z]$, were first studied by W. Danielewski \cite{ori_danielewski} to produce counter examples to Q \ref{q1}. Since then, for a better understanding of non-cancellative surfaces, studies have been undertaken on various generalisations of the Danielewski surfaces, invariants of $\mathbb{G}_a$-actions on these surfaces and their connections with the Cancellation Problem. In \cite{ng_sen_d.danielewski}, N. Gupta and S. Sen studied a collection of surfaces, of the form $\frac{K[X,Y,Z,T]}{(X^dY-P(X,Z),X^eT-Q(X,Y,Z))} $ for some specific polynomials $P\in K^{[2]}$, $Q\in K^{[3]}$ and positive integers $e$, $d$, which they named Double Danielewski surfaces. They have found a subcollection of such surfaces that gives counter examples to the Cancellation Problem. In \cite{pg_ng_gen_danielewski}, P. Ghosh and N. Gupta studied a collection of hypersurfaces of the form $\frac{K[X_1,...,X_n,Y,Z]}{({X_1}^{r_1}...{X_n}^{r_n}Y-F(X_1,...,X_n,Z))} $ for certain $r_i\geq1$ and $F\in K^{[n+1]}$ and have shown these hypersurfaces provide counter examples to the Cancellation Problem. A brief summary of results on various such surfaces can be found in \cite{jauslin_poloni}. 
            
            \indent
             In \cite{veloso}, A. C. Bianchi and M. O. Veloso determined the locally nilpotent derivations of  Danielewski-type surfaces over an algebraically closed field $K$ of characteristic $0$ defined by a more general polynomial of the form $f(X)Y-\varphi(X,Z)$ with some assumptions on $f$ and $\varphi$. In particular they determined the Makar-Limanov invariant and Derksen invariant of such surfaces (defined in Section 2). They also used these invariants to describe the automorphism groups of surfaces defined by polynomials of the form $f(X)Y-\varphi(Z)$.
             
             \indent
             Now results during the past decade, for instance, Neena Gupta's breakthrough result on the Zarisky Cancellation Problem in \cite{ng_inventions}, have shown the desirability of calculating these invariants on hypersurfaces  over fields of arbitrary characteristic.
             
             \indent 
             In this paper, we consider polynomials of the form $f(X)Y-\varphi(X,Z)$ over any field $K$ (not necessarily algebraically closed) of arbitrary characteristic, determine Makar-Limanov invariant and Derksen invariant of surfaces defined by them and using these invariants we prove that a subcollection of such surfaces provides new counter examples to the Cancellation Problem. More precisely, over any field $K$ (not necessarily algebraically closed) of arbitrary characteristic, we study surfaces of the following type:     
              \[A=\frac{K[X,Y,Z]}{(f(X)Y-\varphi(X,Z))}\] where $f=X^r+a_{r-1}X^{r-1}+...+a_1X+a_0$ is a monic polynomial in $X$ and $\varphi(X,Z)=Z^d+c_{d-1}(X)Z^{d-1}+...+c_1(X)Z+c_0(X)$ is such that $d\geq2$, $r\geq2$. In Section 3, we show that the ring of invariants of any non-trivial exponential map on the surface $A$ is $k[x]$, where $x$ is the image of $X$ in $A$  (Theorem \ref{mtheo}). Using this we calculate the Makar-Limanov and Derksen invariants of the ring $A$ (Corollary \ref{mtheinv}). In Section 4, we provide a complete characterization of isomorphisms between such surfaces (Theorem \ref{t6.1} and Theorem \ref*{t4.2}). In Section 5, we show that this collection of surfaces accommodates a subcollection of pairwise non-isomorphic surfaces that are stably isomorphic (Theorem \ref{th5.2}) thus providing a new class of counterexamples to the Cancellation Problem (Q \ref*{q1}). Before discussing the main results we recollect in Section 2 a few definitions and known results that are going to be used in this paper.
                
			\section{Preliminaries}
			\indent
			
			  Throughout this section $K$ will denote a field of arbitrary characteristic $p\geq 0$. We recall the definition of an exponential map on a $K$-algebra, the ring-theoretic formulation of $\mathbb{G}_a$-actions.
			\begin{Defn}
				\rm{Let $K$ be a field and $A$ be a $K$-algebra. Let $\phi: A\longrightarrow A^{[1]}$ be a $K$-algebra homomorphism. For an indeterminate $U$ over $A$, let us denote the map $\phi: A\longrightarrow A[U]$ by the notation $\phi_U$. The map $\phi$ is said to be an exponential map on $A$ if it satisfies the following two properties:}
				
				\begin{enumerate}
				\item[\rm(i)] $\epsilon_0\circ\phi_U $ is identity on $A$, where $\epsilon_0: A[U]\longrightarrow A$  is the evaluation map at $U=0$.
				
				\item[\rm(ii)] $\phi_V\circ\phi_U=\phi_{V+U}$, where $\phi_V: A[U]\longrightarrow A[U,V]$ is the extension of the  homomorphism $\phi_V: A\longrightarrow A[U]$ by $\phi_V(U)=U$.
				\end{enumerate} 
			\end{Defn}
			
			For an element $a\in A$, we define the $\phi$-degree of $a$, denoted by $\text{deg}_{\phi}(a)$, to be the degree of the polynomial $\phi(a)\in A^{[1]}$.
			Let $exp(A)$ denote the collection of all the exponential maps on $A$.
			\begin{Defn}
				\rm{ Let $R$ be an integral domain. For a totally ordered Abelian group $G$, a function $\delta:R\longrightarrow G\cup\{-\infty\}$ is said to be a degree function if for  all $a,b\in R$ :}
				\begin{enumerate}
				\item[\rm(i)] For $a\in R$, $\delta(a)=-\infty \Leftrightarrow a=0$
				\item[\rm(ii)] $\delta(ab)=\delta(a)+\delta(b)$
				\item[\rm(iii)]  $\delta(a+b)\leq max\{\delta(a),\delta(b)\}$.
			\end{enumerate}
			\end{Defn}
			Note that for any $\phi\in exp(A)$, $\text{deg}_{\phi}:A\longrightarrow \mathbb{Z}\cup \{-\infty\}$ is a degree function.
			For $\phi\in exp(A)$ on a $K$-algebra $A$, the subring $A^{\phi}:=\{a\in A: \phi(a)=a\}$ is called the ring of invariants or constants of the map $\phi$.
			\indent
			
			An exponential map is called trivial if $A=A^{\phi}$. Note that an element $a\in A^{\phi}$ iff $\text{deg}_{\phi}(a)\leq 0$.
			\indent
			
			We can also view exponential maps as a sequence of functions. Given an exponential map $\phi:A\longrightarrow A[U]$, let us define functions $\phi^n: A\longrightarrow A$ for each $n\geq 0$, given by $\phi^n(a)=$ the coefficient of $U^n$ in the polynomial $\phi(a)$ for all $a\in A$.
			From the definition of exponential maps it follows that this sequence of functions satisfies the following conditions:
			\begin{enumerate}
			\item[\rm(i)] These functions are $K$-linear.
			\item[\rm(ii)] $\phi^0(a)=a$ for all $a\in A$.
			\item[\rm(iii)] For each $a\in A$, there are only finitely many integers $n$ such that $\phi^n(a)$ is non-zero.
			\item[\rm(iv)] (Leibniz Rule) $\phi^n(ab)=\sum_{i+j=n}\phi^i(a)\phi^j(b)$ for all $n\geq 0 $ and for all $a, b\in A$.
			\item[\rm(v)] (Iterative Property) $\phi^i\phi^j=\binom{i+j}{i} \phi^{i+j}$
		\end{enumerate}
		\indent
		
	We call the collection $D=\{\phi^0,\phi^1,\phi^2,...\}$ the locally finite iterative higher derivation associated with the exponential map $\phi$.

			\begin{Defn}
			\rm	For an affine $k$-domain $A$ the Makar-Limanov invariant is defined by
				$${\rm ML}(A):=\bigcap_{\phi \in exp(A)} A^{\phi}$$
				and the Derksen invariant, denoted by ${\rm DK}(A)$, is defined to be the subalgebra generated by the ring of invariants of all the nontrivial exponential maps on $A$.
			\end{Defn}
			\medskip
			Now we record some standard results about exponential maps which can be found in \cite{rig_small_domain} and \cite{ng_inventions}.
			\begin{lem}\label{l2.4}
				Let $A$ be an affine domain over a field $K$ of characteristic $p\geq 0$. Suppose $\phi: A\longrightarrow A[U]$ is a nontrivial exponential map on $A$. 
				Then the following hold
				\begin{enumerate}
				\item[{\rm(i)}] $ A^{\phi}$ is inert in $A$, i.e, $ab\in A^{\phi}\implies a, b \in A^{\phi} $ for all non-zero $a, b\in A$. In particular, $ A^{\phi}$ is algebraically closed in $A$.
				\item[{\rm(ii)}] If $a\in A$, then $\rm{deg}$$_{\phi}(\phi^i(a))\leq {\rm deg}_{\phi}(a)-i$ for all $i\geq 0$. In particular, if ${\rm deg}_{\phi}(a)=n \geq 0$, then $\phi^n(a)\in A^{\phi}$.  
				\item[{\rm(iii)}] Let $x\in A$ be an element of least positive $\phi$-degree, say $n={\rm deg}_{\phi}(x)$, and let $c$ be the coefficient of $U^n$ in $\phi(x)$ then

				      \item[{\rm(a)}] $\phi^i(x)\in A^{\phi}$ for all $0\leq i\leq n$.
				   \item[{\rm(b)}] $\phi^i(x)=0$ $\forall$ $i>1 $ that is not a power of $p$. 
				   \item[{\rm(c)}] $n \mid{\rm deg}_{\phi}(a)$  $ \forall$ non-zero element $a\in A$, \item[{\rm(d)}]  $A[c^{-1}]=A^{\phi}[c^{-1}][x]$.
				\item[{\rm(iv)}] ${\rm trdeg}_K(A^{\phi})={ \rm trdeg}_K(A)-1$.
				\item[{\rm(v)}] If ${\rm trdeg}_K(A)=1$ and $\Tilde{K}$ is the algebraic closure of $K$ in $A$, then $A=\Tilde{K}^{[1]}$ and $A^{\phi}=\Tilde{K}$
				\item[{\rm (vi)}]Let $S$ be any multiplicative subset of $A^{\phi}\setminus{0}$, then $\phi$ can be extended to an exponential map $S^{-1}\phi$ on $S^{-1}A$ by defining $S^{-1}\phi(a/s):=\phi(a)/s$ for all $a\in A,s\in S$. Further the ring of invariants of  $S^{-1}\phi$ is $S^{-1}(A^{\phi})$.
				\end{enumerate}
			\end{lem}
			\smallskip
			
			\begin{Defn}
				\rm A collection of $K$-linear subspaces $\{A_n\}_{n\in \mathbb{Z}}$ of an affine $K$-domain $A$ is called a proper $\mathbb{Z}$-filtration if it satisfies the following conditions:
				\begin{enumerate}
				\item[(i)] $A_n\subseteq A_{n+1}$ for all n in $\mathbb{Z}$,
				\item[(ii)] $A=\bigcup_{n\in \mathbb{Z}} A_n$,
				\item[(iii)] $\bigcap_{n\in \mathbb{Z}} A_n={0}$ and 
				\item[(iv)] $(A_n\setminus A_{n-1}). (A_m\setminus A_{m-1})\subseteq A_{n+m}\setminus A_{n+m-1}$ for all $n, m\in\mathbb{Z}$.
				\end{enumerate}
			\end{Defn}
			\smallskip
			A proper $\mathbb{Z}$-filtration $\{A_n\}_{n\in \mathbb{Z}}$ of $A$ is called admissible if there exists a finite generating set $\Gamma$ of $A$ such that, for any $n\in \mathbb{Z}$ and $a\in A_n$, $a$ can be written as a finite sum of monomials in elements of $\Gamma$ and each of these monomials is an element of $A_n$.
			\\ Given a proper $\mathbb{Z}$-filtration $\{A_n\}_{n\in \mathbb{Z}}$ of an affine $K$-domain $A$, we have an associated $\mathbb{Z}$-graded integral domain $$gr(A):=\bigoplus_{i\in \mathbb{Z}}\frac{A_i}{A_{i-1}}$$
			and a map $\rho: A \longrightarrow gr(A)$ defined by $\rho(a)= a+ A_{n-1}$, if $a\in A_n\setminus A_{n-1}$. 
			Below we note a few remarks.
			\smallskip
			\begin{rem}
			\rm
			\begin{enumerate}	
				\item[(i)] The map $\rho$ is multiplicative but not additive in general.
				\smallskip
				\item[(ii)] If $\{A_n\}_{n\in \mathbb{Z}}$ is a proper admissible $\mathbb{Z}$-filtration on $A$ with a finite generating set $\Gamma$, then $gr(A)$ is generated by $\rho(\Gamma)$.
				\smallskip
				\item[(iii)] If $A=\bigoplus_{n\in \mathbb{Z}}A_n$ is a $\mathbb{Z}$-graded affine $K$-domain then the filtration $\{B_i\}_{i\in \mathbb{Z}}$ defined by $B_n:=\bigoplus_{i\leq n}A_n$ is a proper admissible $\mathbb{Z}$-filtration on $A$, and the associated graded domain $gr(A)=\bigoplus_{i\in \mathbb{Z}}\frac{B_i}{B_{i-1}}\cong \bigoplus_{i\in \mathbb{Z}}A_i\cong A$ with this identification, the image $\rho(a)$ of a non-zero element $a\in A$, under the map $\rho: A \longrightarrow gr(A)\cong A$,  is the highest degree homogeneous component of $a$.
				\end{enumerate}
				
			\end{rem}
			\smallskip
			\begin{Defn}
			\rm	An exponential map $\phi$ on a graded $K$-domain $A$ is said to be homogeneous if $\phi: A\longrightarrow A[U]$ is a homogeneous ring homomorphism of graded rings where $A[U]$ has a grading induced from the grading of $A$ such that $U$ becomes a homogeneous element. 
			\end{Defn}
			
			The following version of a result on homogenization of exponential maps due to H. Derksen, O. Hadas and L. Makar-Limanov \cite{hadas-derksen} is in (\cite{craciolla_russelkorus}, Theorem 2.6)  .
			\begin{thm}\label{hadas-derk}
				\smallskip
				Let A be an affine $K$-domain with an admissible proper $\mathbb{Z}$-filtration and $gr(A)$ be the associated domain. Let $\phi$ be a non-trivial exponential map on $A$. Then $\phi$ induces a non-trivial homogeneous exponential map $\overline{\phi}$ on $gr(A)$ such that $\rho(A^{\phi})\subseteq gr(A)^{\overline{\phi}}$.
			\end{thm}

			\begin{Defn}
				\rm
				Over a field $F$, an element $g\in F[X,Y]$ is said to be a
				\begin{enumerate}
					\smallskip
					\item[(i)]coordinate if $F[X,Y]= F[g]^{[1]}$,
					\smallskip
					\item[(ii)]line if $\frac{F[X,Y]}{(g)}\cong F^{[1]}$,
					\smallskip
					\item[(iii)]a non-trivial line if it is a line but not a coordinate.
				\end{enumerate} 
			
			\end{Defn}
			Now we recall a consequence of the Epimorphism Theorem of S. S Abhyankar and T. T Moh, (\cite{abh_moh}), as presented in (\cite{abh_moh_tifr}, Corollary 9.26, and Remark 9.29).
			\begin{thm}\label{abhmoh}
				\smallskip
				Let $K$ be a field of characteristic $p\geq0$. Let $h(X,Y)\in K[X,Y]$ be such that $\frac{K[X,Y]}{(h)}\cong_{K}K^{[1]}$ and $p\nmid {\rm gcd}({\rm deg}_X(h),{\rm deg}_Y(h))$. Then ${\rm deg}_X(h)\mid {\rm deg}_Y(h)$ or ${\rm deg}_Y(h)\mid {\rm deg}_X(h)$.
				
			\end{thm}
			For the convenience of the reader, we record below a few elementary observations.
			
			\begin{lem}\label{2.11}
				Let $E$ and $F$ be integral domains such that $E\subset F$. If there exists an $a\in E$ such that $E[a^{-1}]=F[a^{-1}]$ and $aF\cap E = aE$ then $E=F$.
			\end{lem}
			
			\begin{lem}\label{l1}
				Let $R$ be a unique factorization domain and $a,b\in R$ be such that ${\rm gcd}(a,b)=1$. Then $\frac{R[Y]}{(aY-b)}$ is an integral domain.
				
			\end{lem}
			
		\begin{lem}\label{l3}
				Let $B$ be an integral domain and $a,b,c\in B\setminus\{0\}$. If $R=\frac{B[Y]}{(acY-b)}$, then $aR \cap B= (a,b)B$.  
			\end{lem}

			\section{Main Results}
			\indent
			
				For convenience we record a few elementrary observations. The results in Lemma \ref*{l3.1} below have been stated in  (\cite{veloso}, Proposition 4) over fields of characteristic zero.
			\begin{lem}\label{l3.1}
				\medskip
				Let $K$ be any field of arbitrary characteristic. Let $A=\frac{K[X,Y,Z]}{(f(X)Y-\varphi(X,Z) )}$ where $f=X^r+a_{r-1}X^{r-1}+...+a_1X+a_0$ is a monic polynomial in $X$ and $\varphi(X,Z)=Z^d+c_{d-1}(X)Z^{d-1}+...c_1(X)Z+c_0(X)$ is such that $d\geq2$, $r\geq2$. Let $x,y,z$ denote respectively images of $X,Y,Z$ in $A$. Then the following hold:
				\begin{enumerate}
					\item[{\rm(i)}] A is an integral domain and $\rm{trdeg}\it_K (A)=2$.
					\smallskip
					\item[{\rm(ii)}] $K[x]$
					is inert in $A$ and $K(x)\cap A=K[x]$.
					\smallskip
					\item[{\rm(iii)}] If $S=K[x]\setminus \{0\}$, then $S^{-1}A=K(x)[z]$.
					\smallskip
					\item[{\rm(iv)}] Any $g\in A$ can be expressed uniquely as
					\begin{equation}
						g=g_0(x,z)+g_1(x,z)y+...+g_m(x,z)y^m
					\end{equation} 
					 for some polynomials $g_i$'s in $K[X,Z]$ such that $0\leq {\rm deg}_Z(g_i(X,Z))\leq d-1$ for all $i$,  $0\leq i\leq m$ and $m\geq 0$.
				\end{enumerate}   
			\end{lem}
			\begin{proof}
				
				\smallskip
				(i) Follows from Lemma \ref{l1}.
				\newline
				\smallskip
				(ii) The proof given in (\cite{veloso}, Proposition 4) is independent of the characteristic of $K$.
				\newline
				\smallskip
				(iii) Follows directly from the definition of $A$.
				\newline
				(iv) The same proof given in (\cite{veloso}, Proposition 4) works.
				\end{proof} 
				Moreover we have :
				\smallskip
				\begin{lem}\label{l3.2}
					Under the hypotheses of Lemma 3.1 the following hold:  
						\item[{\rm(i)}]There exists a non-trivial exponential map $\phi$ on $A$ such that $A^{\phi}=K[x]$; in particular, ${\rm ML}(A)\subseteq K[x].$
					\smallskip
					\item[{\rm(ii)}] $A \cap (K(x)\oplus K(x)z)=K[x]\oplus K[x]z$. 
					\end{lem}
					\begin{proof}
				(i) Define a map  $\phi : A\longrightarrow A[U]$ by 
				\begin{alignat*}{3}
					&\phi(x)&=&x\\
					&\phi(z)&=&z+f(x)U\\
					&\phi(y)&=&\frac{\varphi(x,z+f(x)U)}{f(x)} =y+ U\alpha(x,z,U) \text{ for some $\alpha \in K^{[3]}$}.
				\end{alignat*}
				\indent
				\smallskip
				Then it is easy to check that $\phi$ is an exponential map such that $K[x]\subseteq A^{\phi} $. Since $K[x]$ is inert in $A$ by Lemma \ref{l3.1}(ii), it is algebraically closed in $A$ and since $\text{trdeg}_KK[x]=\text{trdeg}_KA^{\phi}=1$ we get that $A^{\phi}=K[x]$. 
				\newline
				\smallskip
				(ii) Suppose $g$ is in $A \cap (K(x)\oplus K(x)z)$.
				Then $g=\alpha(x)+\beta(x)z$ for some $\alpha, \beta \in K(x)$. From Lemma \ref{l3.1}(iv), we have
				\[g=\sum_{\substack{0\leq j\\0\leq i\leq  d-1}}a_{ij}(x)z^iy^j\] for some $a_{ij}\in k[x]$. Using the relation $y=\frac{\varphi(x,z)}{f(x)}$, we have the following identity in $K(x)[z]$,
				\[g=\alpha(x)+\beta(x)z=\sum_{\substack{0\leq j\\0\leq i\leq  d-1}}\frac{a_{ij}(x)}{f(x)^j}z^i\varphi(x,z)^j.\] Since $x,z$ are algebraically independent over $K$, and $\varphi$ is monic in $z$ and  $d={\rm deg}_Z(\varphi)\geq 2$, we get $a_{ij}=0$ for $j\geq 1$ and $i\geq 2$. Thus $A \cap (K(x)\oplus K(x)z)\subset K[x]\oplus K[x]z$. The other side inclusion is trivial. Hence we get $A \cap (K(x)\oplus K(x)z)=K[x]\oplus K[x]z$.  
			\end{proof}
			
			\begin{prop}\label{supthm}
				\medskip
				Let $K$ be any field of arbitrary characteristic. Suppose $B=\frac{K[U,V,W]}{(f(U)V-W^d)}$ where $f(U)=U^r+a_{r-1}U^{r-1}+...+a_1U+a_0$ is a monic polynomial in $U$, $d\geq2,r\geq2$ and $u, v, w$ are the images in $B$ of $U, V, W$  respectively . Consider the graded structure on $B$ such that $wt(u)=0, wt(v)=d, wt(w)=1$. Then there does not exist any nontrivial homogeneous exponential map $\phi$ on $B$ such that $v\in B^{\phi}$.
			 
			\end{prop}
			\begin{proof} Suppose, if possible, there exists a nontrivial homogeneous exponential map $\phi : B\longrightarrow B[T]$ such that $v\in B^{\phi}$. We first show that in this case $B^{\phi}=k[v]$. 
				\smallskip
				  As $\phi$ is nontrivial and $B^{\phi}$ is inert in $B$, $v\in B $, neither $u $ nor $w$ can be in $B^{\phi}$ and we have $B^{\phi}\cap K[u]=K$ and $B^{\phi}\cap K[w]=K$.
				  \indent
				  
				   Let $B=\oplus_{i\geq 0}B_i$ be the given graded structure. We know that $B^{\phi}$ is a graded subalgebra of $B$. Let $R=B^{\phi}=\oplus_{i\geq 0}R_i$, where $R_i=R\cap B_i$. From the specified weights of $u,v,w$, it is easy to check $R_0=K$, $R_1=0$,..., $R_{d-1}=0$, $R_d=Kv$. Then as $B^{\phi}$ is inert in $B$, from induction and using the relation $w^d=f(u)v$ whenever needed, we get that for each n, $R_{nd}=Kv^{n}$ and, for $1\leq i\leq (d-1)$, $R_{nd+i}=0$. Therefore $B^{\phi}=K[v]$.
				\indent
				
			     We now show that if such a $\phi$ exists then ${\rm gcd}(r,d)=1$. Let ${\rm deg}_{\phi}(u)=n$, ${\rm deg}_{\phi}(w)=m$. As $\phi^n(u), \phi^m(w) \in B^{\phi}=K[v]$, we get that
			    $\phi(u)=u+...+g(v)T^n $ for some $g(v)\in K[v]$ and $\phi(w)=w+...+h(v)T^m  $ for some $ h(v)\in K[v]$. Let $\alpha= {\rm deg}_v(g(v))$ and $\beta={\rm deg}_v(h(v))$.
			    As $\phi$ is homogeneous, we have $ {\rm grdeg}(u)={\rm grdeg}(g(v)T^n)$ and ${\rm grdeg}(w)={\rm grdeg}(h(v)T^m))$ implying that $g(v)=\lambda v^{\alpha}$ and $h(v)=\mu v^{\beta}$, for some $\lambda, \mu\in K^*$ . From the relation $f(u)v=w^d$, we get 
			    \begin{alignat*}{3}
			    	&\phi(f(u))v=\phi(w^d)\\
			    	\Rightarrow &\phi(f(u))v=\phi(w)^d\\
			    	\Rightarrow &(a_0+a_1\phi(u)+...+\phi(u)^r)v=\phi(w)^d.
			    \end{alignat*}
			    
			    Comparing the leading $T$-degree terms on both sides, we get	\begin{alignat*}{3}
			    	&v(g(v)^rT^{nr})&=&h(v)^dT^{md}\\
			    	\Rightarrow &\lambda^r v^{1+\alpha r }T^{nr}&=&\mu^d v^{\beta d}T^{md}\\
			    	\Rightarrow &\alpha r+1&=&\beta d\\
			    	\Rightarrow &{\rm gcd}(r,d)&=&1.
			    \end{alignat*}

				 By Lemma \ref*{l2.4}, we can extend $\phi$ to a nontrivial exponential map on the localised ring $\hat{B}=\frac{K(V)[U,W]}{(f(U)V-W^d)}$. But as ${\rm trdeg}_{K(v)}(\hat{B})=1$, $\hat{B}$ must be a polynomial ring over a field by Lemma \ref{l2.4}, in particular, regular. If $f$ has at least one double root, say $a$, then $\hat{B}_{(u-a,w)}$ could not be a regular local ring, a contradiction. Thus all the roots of $f$ are simple roots. But then using Eisenstein's criterion, we get that $\hat{B}$ is a geometrically integral domain over $K(V)$. Hence we must have $\hat{B}=K(V)^{[1]}$. Therefore, by Theorem \ref{abhmoh}, either $r\mid d$ or $d\mid r$,  contradicting that ${\rm gcd}(r,d)=1$. Thus there does not exist any nontrivial homogeneous exponential map on $B$ such that $v\in B^{\phi}$.
			\end{proof}
			\medskip
			
			\begin{lem}\label{RedLem}
				Let $K$ be any field of arbitrary characteristic, $f=X^r+a_{r-1}X^{r-1}+...+a_1X+a_0$ is a monic polynomial in $K[X]$ with degree $r\geq2$, $\varphi(X,Z)=Z^d+c_{d-1}(X)Z^{d-1}+...c_1(X)Z^1+c_0(X)$ is a polynomial in $K[X,Z]$ such that $d\geq2$ and $A=\frac{K[X,Y,Z]}{(f(X)Y-\varphi(X,Z) )} $. Let $x,y,z$ be the images of $X,Y,Z$ respectively in $A$. Then there exists a proper admissible $\mathbb{Z}$-filtration $\{A_i\}_{i\in \mathbb{Z}}$ on $A$ such that $x\in A_0\setminus A_{-1}$, $y\in A_{d}\setminus A_{d-1}$ and $z\in A_{1}\setminus A_0$. Further, if $B$ is the associated graded algebra, then $B\cong\frac{K[U,V,W]}{(f(U)V-W^d)}$.
			\end{lem}
			\begin{proof}
					 We note that $A$ is a subring of $D=K[x,f(x)^{-1},z]$ and $D$ has a natural grading as $D=\bigoplus_{i\geq 0}K[x,f(x)^{-1}]z^i=\bigoplus_{i\in \mathbb{Z}}D_i$ where $D_i=0$ for all $i<0$ and $D_i= k[x,f(x)^{-1}]z^i$ for all $i\geq 0$. Now as D is a finitely generated graded $K$-domain, the filtration $\widetilde{D}_i$ defined by $\widetilde{D}_i=\bigoplus_{j\leq i}D_j$ is a proper admissible $\mathbb{Z}$-filtration on D. So $A_i=\widetilde{D}_i\cap A$ is a proper $\mathbb{Z}$-filtration on $A$. As any $g$ in $A$ can be expressed uniquely as $g=g_{0}(x,z)+g_1(x,z)y+...+g_m(x,z)y^m$   for some polynomials $g_i$'s in $K[X,Z]$ such that $0\leq {\rm deg}_Z(g_i(X,Z))\leq d-1$ for all $i$, $0\leq i\leq m$ and $m\geq 0$, we get that this filtration $\{A_i\}_{i\in\mathbb{Z}}$ is a proper admissible filtration with finite generating set $\Gamma=\{x,y,z\}$. Clearly $x\in A_0\setminus A_{-1}$, $y\in A_{d}\setminus A_{d-1}$ and $z\in A_{1}\setminus A_0$. For an element $a$ of $A$, let $\overline{a}$ denote the image of $a$ under the map $\rho:A\longrightarrow B$. Then $B$ is generated by $\overline{x},\overline{y} ,\overline{z}$.
				\indent
				
				Now we prove that $B$ is isomorphic to $\frac{K[U,V,W]}{(f(U)V-W^d)}$, where $u=\overline{U},v=\overline{V},w=\overline{W} $ have been identified with $\overline{x},\overline{y} ,\overline{z}$ respectively.
				Define a surjective map $\eta : K[U,V,W]\longrightarrow B$ by $\eta(U)=\overline{x},\eta(V)=\overline{y}, \eta(W)=\overline{Z}$. From the relation $f(x)y=\varphi(x,z)$ in $A$ we get that $\overline{f(x)y}=\overline{f(x)}\overline{y}=f(\overline{x})\overline{y}=\overline{z^d}=\overline{z}^d=\overline{\varphi(x,z)}$ in $B$.
				Thus $\eta(f(U)V-W^d)=0$ and as  ${\rm trdeg}_K(B)=2$  we have an isomorphism $$\overline{\eta}:\frac{K[U,V,W]}{(f(U)V-W^d)}\cong_{K} B.$$
			\end{proof}
			Now we prove our main theorem.
			\begin{thm}\label{mtheo}
				\smallskip
			Under the hypotheses of Lemma \ref{RedLem}, ring of invariants of any non trivial exponential map $\phi$ on $A$ is $k[x]$.  
			\end{thm}
			\begin{proof}
				\smallskip

				Let $\phi$ be a nontrivial exponential map on $A$. Any $g\in A$ can be written as  $g=g_{0}(x,z)+g_1(x,z)y+...+g_m(x,z)y^m$   for some polynomials $g_i$'s in $K[X,Z]$ such that $0\leq {\rm deg}_Z(g_i(X,Z))\leq d-1$ for all $i$, $0\leq i\leq m$ and $m\geq 0$ from Lemma \ref{l3.1}(iv). We consider the degree induced from the filtration $\{A_i\}_{i\in \mathbb{Z}}$. Let $\hat{g}$ denote the highest degree summand of $g$ in the expansion $g=g_{0}(x,z)+g_1(x,z)y+g_1(x,z)y+...+g_m(x,z)y^m$.
				Let $\overline{g}$ be as defined in the proof of Lemma \ref{RedLem}. Clearly $\rho(g)=\rho(\hat{g})$. Suppose $g\in A^{\phi}\setminus K$. By Lemma \ref{RedLem} and Theorem \ref{hadas-derk}, $\phi$ induces a non trivial exponential map $\overline{\phi}$ on $B=\frac{K[U,V,W]}{(f(U)V-W^d)}$ with  $\overline{g}\in B^{\overline{\phi}}$ (from the proof of Lemma \ref{RedLem}, we know images of $U,V,W$ in $B$ have been identified with $\overline{x},\overline{y} ,\overline{z}$ respectively). It is enough to show that $g\in K[x]$. 
				 We note that if $g\in K[x,z]$, i.e $m=0$, then $g$ must be in $K[x]$. For if $g\in K[x,z]$ then $\hat{g}=x^iz^j$. If $j=0$, then $g\in K[x]\setminus K$ and we are done. Otherwise if $j>0$, then as $\overline{g}=\overline{\hat{g}}\in B^{\overline{\phi}}$, we get $\overline{z}\in B^{\overline{\phi}}$. Using the relation $f(\overline{x})\overline{y}=\overline{z}^d$ and the fact that $B^{\overline{\phi}}$ is factorially closed, we get  $\overline{x},\overline{y},\overline{z}\in B^{\overline{\phi}}$ and hence $\overline{\phi}$ is trivial which is a contradiction. Hence $j=0$ and $g\in K[x]\setminus K$. 
				Suppose that $m\geq 1$.
				First we note that whenever $0\leq j_1, j_2\leq (d-1)$, $${\rm Deg}(\lambda_1(x)y^{i_1}z^{j_1}) ={\rm Deg}(\lambda _2(x)y^{i_2}z^{j_2})\Rightarrow i_1=i_2, j_1=j_2.$$
				\indent
				Thus $\hat{g}=\lambda(x)y^mz^j$ for some $j\geq0$ and $\lambda(x)\in K[x]$. As $m\geq1$, $B^{\overline{\phi}}$ is inert and by Theorem \ref{hadas-derk}, $\overline{g}=\overline{\hat{g}}\in B^{\overline{\phi}}$, we have $\overline{y}\in B^{\overline{\phi}}$. Thus from Proposition \ref{supthm} we reach a contradiction. Hence $m=0$. And thus $A^{\phi}=K[x]$.
			\end{proof}
			From Theorem \ref*{mtheo} and Lemma \ref{l3.2}, we get the following result:
			\begin{cor}\label{mtheinv}
				\medskip
				 Let $K$ be any field of arbitrary characteristic and $A=\frac{K[X,Y,Z]}{(f(X)Y-\varphi(X,Z) )} $ with $f=X^r+a_{r-1}X^{r-1}+...+a_1X+a_0$ and $\varphi(X,Z)=Z^d+c_{d-1}(X)Z^{d-1}+...+ c_1(X)Z^1+c_0(X)$ such that $d\geq2, r\geq2$ and let $x$ denote the image of $X$ in $A$. Then ${\rm ML}(A)=K[x]$ and ${\rm DK}(A)=K[x]$.
			\end{cor}
			\section{Isomorphisms of Danielewski surfaces}
			\indent
			
			In this section we characterize isomorphisms of General Danielewski surfaces. Let $K$ be any field of arbitrary characteristic. Let us consider two surfaces $A_1$ and $A_2$, where 
			$$A_1=\frac{K[X,Y,Z]}{(f(X)Y-\varphi_1{(X,Z)})}$$
			and 
			$$A_2=\frac{K[X,Y,Z]}{(g(X)Y-\varphi_2{(X,Z)})}$$
			such that $\varphi_1$, $\varphi_2$ are monic in $Z$ with $d_1:={\rm deg}_Z(\varphi_1)\geq 2$, $d_2:={\rm deg}_Z(\varphi_2)\geq 2$ and ${\rm deg}(f)$, ${\rm deg}(g)\geq 2$. Let $x_i$, $y_i$, $z_i$ denote the images of $X,Y,Z$ in $A_i$  respectively for $i=1,2$. 
			\medskip
			\begin{thm}\label{t6.1}
				Let $T: A_1\longrightarrow A_2$ be a $K$-algebra isomorphism. Let $f(x_1)={p_1}^{\alpha_1}...{p_n}^{\alpha_n}$ and $ {g(x_2)={q_1}^{\beta_1}...{q_m}^{\beta_m}}$ be the prime factorizations of $f,g$ in $K[x_1], K[x_2]$ respectively, where $\alpha_i$, $\beta_j\geq 1$ for all $i,j$. Then the following hold:
				\begin{enumerate}
					\smallskip
				\item[{\rm(i)}] $T(x_1)=\lambda x_2+\mu$,  $T(z_1)=\gamma z_2+\delta $  for some $\lambda,\gamma \in K^*$, $\mu \in K$, and $\delta \in K[x_2]$. Moreover $n=m$.
		             \smallskip
				\item[{\rm(ii)}]  $T(p_i)=\mu_{ij} q_j $ for some $1\leq j\leq n$, $ \mu_{ij}\in K^*$ for each $i$, $1\leq i\leq n$ and $\alpha_i=\beta_j$. 
				     \smallskip
				 \item[{\rm(iii)}] $T(f(x_1))=ug(x_2)$ for some unit $u$ in $K$.
				     \smallskip
				 \item[{\rm(iv)}] $d_1=d_2$.
				  \smallskip
				  \item[{\rm(v)}]  $T((f(x_1),\varphi_1(x_1,z_1))K[x_1,z_1])=(g(x_2),\varphi_2(x_2,z_2))K[x_2,z_2]$. 
				   \item[{\rm(vi)}]
				   \smallskip
				   $T(\varphi_1(x_1,z_1))=\varphi_1(\lambda x_2+\mu,\gamma z_2+\delta)=\gamma^{d_2} \varphi_2(x_2,z_2)+g(x_2)\theta(x_2,z_2)$ for some $\theta \in K[x_2,z_2]$ with ${\rm deg}_{z_2}(\theta)\leq d_2-1$.
				      \smallskip
				 \item[{\rm(vii)}] $T(y_1)=u^{-1}\gamma^{d_2} y_2+ \nu $ where $ \nu= u^{-1}\theta$. 
				   \end{enumerate}
			
		\medskip
		    
				
			\end{thm}
		
            \begin{proof}
            	\smallskip
            	    By Corollary \ref{mtheinv}, we get that $T(K[x_1])=K[x_2]$ and hence $T(x_1)=\lambda x_2+\mu$ for some $\lambda \in K^*$ and $\mu\in K$. Extending this isomorphism to the localization of $A_1$ at $K(x_1)\setminus \{0\}$, from part (iii) of Lemma \ref{l3.1}, we have $K(T(x_1))[T(z_1)]=K(x_2)[z_2]$. Hence we get $T(z_1)=\gamma z_2+\delta $ for some $\gamma ,\delta \in K(x_2)$. By Lemma \ref{l3.2} (ii), $\gamma ,\delta \in K[x_2]$. Similarly using $T^{-1}$, we conclude that $\gamma$ is a unit in $K$. In particular, we have $$ T(K[x_1,z_1])=K[x_2,z_2].$$ Now there is an exponential map $\psi$ on $A_2$ such that $\psi(x_2)=x_2$, $ \psi(z_2)=z_2+g(x_2)U$. From Theorem \ref{mtheo}, we get that $A_2^{\psi}= K[x_2]=T(K[x_1])$. Then applying $\psi$ on both side of the the relation $T(f(x_1))T(y_1)=\varphi_1(T(x_1),T(z_1))$ and comparing the coefficients of $U^{d_1}$ on both sides we get, we get that $T(f(x_1))\psi^{d_1}(T(y_1))=\gamma^{d_1}g(x_2)^{d_1}$. Thus we get that for each prime factor $p_i$ of $f$ in $K[x_1]$, $T(p_i)$ is one of the prime factors $q_1,q_2,...,q_n $ of $g$ in $K[x_2]$ i.e $T(p_i)=\mu_{ij} q_j$ for some $1\leq j\leq n$, and $\mu_{ij}\in K^*$. And similarly using the isomorphism $T^{-1} $ from $A_2$ to $A_1$ we get that number of distinct prime factors of $f(W),g(W)$ in $K[W]\cong K^{[1]}$ are same i.e $n=m$. This proves part (i) of the theorem.
            	  
            	  \indent
            	   Now we may write ${g(x_2)={q_1}^{\beta_1}...{q_n}^{\beta_n}}$ such that $T(p_i)=q_i$ for all $i$. Suppose if possible, $\alpha_1 >\beta_1$. As $T(K[x_1,z_1])=K[x_2,z_2]$, we get that $T({p_1}^{\alpha_1}A_1\cap K[x_1,z_1])={q_1}^{\alpha_1} A_2\cap K[x_2,z_2]$. Using Lemma \ref{l3}, we get \[T(({p_1}^{\alpha_1},\varphi_1(x_1,z_1))K[x_1,z_1])=({q_1}^{\alpha_1},{q_1}^{\alpha_1-\beta_1}\varphi_2(x_2,z_2))K[x_2,z_2].\] Thus we get that $T(\varphi_1(x_1,z_1))=\varphi_1(\lambda x_2+\mu, \gamma z_2+\delta)\in (q_1)K[x_2,z_2]$. This contradicts the fact that $\varphi_1$ is monic in $z_1$. Hence $\alpha_1\leq \beta_1$. Similarly we get that $\beta_1\leq \alpha_1$ and so $\alpha_1=\beta_1$. Using the same argument for all $i$, $1\leq i \leq n$, we get that $\alpha_i=\beta_i$. This proves (ii). And (iii) follows directly from (i) and (ii).
            	   
            	  \indent
            	  Now we prove that if $A_1$ and $A_2$ are isomorphic then $d_1=d_2$. Without loss of generality we may assume that $d_1 < d_2$. From part (i) and (iii), we get  
            	  
            	  \begin{equation}
            	  ug(x_2)T(y_1)=\varphi_1(\lambda x_2+\mu,\gamma z_2+\delta)	
            	  \end{equation}

            	   \medskip
            	  Let 
            	  \begin{equation}
            	  	T(y_1)=g_0(x_2,z_2)+g_1(x_2,z_2)y_2+...+g_l(x_2,z_2){y_2}^l
            	  \end{equation}
            	   
            	  be the unique expression of $T(y_1) $ in $A_2$ where ${\rm deg}_{z_2}(g_i)\leq d_2-1 $ for all $0\leq i\leq l$. As $T$ is an isomorphism and $T(K[x_1,z_1])=K[x_2,z_2]$ we must have $l\geq 1$. As $d_1<d_2$, right side of (2) has no $y_2$ term but left side of (2) has $y_2$ term. But this is a contradiction. Thus $d_1\geq d_2.$  Similarly we can show that $d_2\geq d_1$. Hence $d_1=d_2=d$(say) and (iv) is proved.
            	  
            	  \smallskip
            	   (v) follows from (iii) and Lemma \ref*{l3}.
            	   \smallskip
            	   
            	   It follows from (v)  that,
            	   \begin{equation}
            	   	\varphi_1(\lambda x_2+\mu,\gamma z_2+\delta)=\zeta \varphi_2(x_2,z_2)+g(x_2)\theta'(x_2,z_2)\text{ for some $\zeta$, $\theta' \in K[x_2,z_2]$.}
            	   \end{equation}
            	  
            	  But since $\varphi_2$ is monic in $z_2$ and $d_1=d_2=d$ we get $$ \zeta-\gamma^d\in (g(x_2))K[x_2,z_2].$$  
            	  Thus replacing $\zeta$ with $\gamma^d$ in (4), we get there exists $\theta \in K[x_2,z_2]$ such that  
            	  \[\varphi_1(\lambda x_2+\mu,\gamma z_2+\delta)=\gamma^{d_2} \varphi_2(x_2,z_2)+g(x_2)\theta(x_2,z_2). \]
            	  And it is clear that  ${\rm deg}_{z_2}(\theta)\leq d_2-1$. This proves part (vi).

            	  
            	  \indent
            	  Now from (iv) and equation (2) above it is easy to see that $l=1$ in equation (3). Thus after expanding (2) we get,
            	  \begin{equation}
            	  	ug(x_2)(g_0(x_2,z_2)+g_1(x_2,z_2)y_2)=\gamma^{d} \varphi_2(x_2,z_2)+g(x_2)\theta(x_2,z_2)=g(x_2)(\gamma^{d} y_2+\theta).
            	  \end{equation} 
            	  Now as  ${\rm deg}_{z_2}(\theta)\leq d-1$, from (5) it follows that $g_0(x_2,z_2)=u^{-1}\theta$ and $ g_1(x_2,z_2)=u^{-1}\gamma^d$. Now taking $\nu=g_0$, we get $T(y_1)=u^{-1}\gamma^{d} y_2+ \nu $ such that $ \nu= u^{-1}\theta$. This proves (vii).

            	  \indent
            	  
             
            	
            \noindent	 

            \end{proof}
            \indent
             \begin{thm}\label{t4.2}
            	Let $A_1$, $A_2$ be as in Theorem \ref{t6.1}. Let $T: A_1\longrightarrow A_2$ be a homomorphism such that  {\rm(i)},  {\rm(iv)} of Theorem \ref{t6.1} are satisfied then the following are equivalent:
            		\begin{enumerate}
            		\smallskip
            		\item[{\rm(I)}] T is an isomorphism.
            		\smallskip
            		\item[{\rm(II)}] $T((f(x_1),\varphi_1(x_1,z_1))K[x_1,z_1])=(g(x_2),\varphi_2(x_2,z_2))K[x_2,z_2]$.
            		\smallskip
            		\item[{\rm(III)}]  $T(f(x_1))=ug(x_2)$ for some unit $u$ in $K$.
            		
            	\end{enumerate} 
            \end{thm}
            \begin{proof}
            	(I) $\Rightarrow$ (II) Follows from Theorem \ref{t6.1}.
            	
            	\smallskip
            	(II) $\Rightarrow$ (III)  It is clear from (i), that $T$ is an isomorphism from $k[x_1,z_1]$ to $K[x_2,z_2]$. Now (II) implies
            	\begin{equation}
            		T(f(x_1))=\zeta' g(x_2)+ \tau\varphi_2(x_2,z_2) \text{ for some $\zeta$ and $\tau$ $\in$ $K[x_2,z_2]$. }
            	\end{equation}
            	
            	But from (i), $T(f(x_1))\in K[x_2]$ and $\varphi_2$ is monic in $z_2$, so $\tau \in (g(x_2))K[x_2,z_2]$ and hence we get from (6), \[ T(f(x_1))=\zeta g(x_2) \text{  for some $\zeta\in K[x_2,z_2]$}.\] From (i) $n=m$,  and $T(f(x_1))\in K[x_2]$, so $\zeta \in K$ is a unit. Hence  $T(f(x_1))=ug(x_2)$ for some unit $u\in K$.
            	
            	\smallskip(III) $\Rightarrow$ (I)  As $A_1$, $A_2$ are Noetherian domains of same dimension, it is enough to show that $T$ is surjective.
            	 From equation (2) and (iv), it follows that $l=1$ in equation (3).
            	 After expanding (2), we get, 
            	 
            	    \[
            	 		ug(x_2)(g_0(x_2,z_2)+g_1(x_2,z_2)y_2)=\gamma^{d} \varphi_2(x_2,z_2)+ \theta'(x_2,z_2) 
            	 	\]
            	  for some $\theta' \in K[x_2,z_2]$ with ${\rm deg_{z_2}}(\theta')\leq d-1$.\\
            	 From Lemma \ref{l3}, $\theta'\in (g(x_2),\varphi_2(x_2,z_2))K[x_2,z_2]$. As $\varphi_2$ is monic in $Z$, it follows $\theta'=(g(x_2))K[x_2,z_2]$. Now using same arguments used to prove (vii) of Theorem \ref{t6.1}, we get that  
            	\[T(y_1)=u^{-1}\gamma^{d} y_2+ u^{-1}\theta    \text{    for some $\theta$ in $K[x_2,z_2]$ }.\]
            	This proves that $y_2$ is in the image of $T$ and hence $T$ is surjective. Thus $T$ is an isomorphism. 
            	
            \end{proof}
            
            Let $B=\frac{F[X,Y,Z]}{(h(X)Y-\eta(Z))}$, where $F$ is an algebraically closed field of characteristic $0$, ${\rm deg}(h)\geq 2$, ${\rm deg}(\eta)\geq2$ and $h$ has at least one non-zero root. Then for any automorphism $T$ of $B$, in \cite{veloso}, Bianchi and Veloso have proved that $T(x)=\lambda x$ for some $\lambda \in F^*$, where $x$ is the image of $X$ in $B$. Over an arbitrary field $K$ of arbitrary characteristic, we have the following corollary of Theorem \ref{t6.1} giving a similar result. 	
        \begin{cor}\label{c6}
        	\smallskip
        	 Let $K$ be any field of arbitrary characteristic and $A=\frac{K[X,Y,Z]}{(f(X)Y-\varphi{(X,Z)})}$ where $\varphi$ is monic in $Z$ and ${\rm deg}_Z(\varphi)\geq2$. Suppose, $f(X)=X^ng(X)$ where $n\geq2$ and $g$ has no root of multiplicity $n$. Then for any automorphism $T$ of $A$, $T(x)=\lambda x$ for some nonzero constant $\lambda$, where $x$ is the image of $X$ in $A$.
          \end{cor}
          
          \begin{cor}\label{c6.3}
          	\smallskip
          	Let $K$ be any field of arbitrary characteristic and   $A=\frac{K[X,Y,Z]}{(f(X)Y-\varphi{(X,Z)})}$  where $\varphi$ is monic in $Z$ such that ${\rm deg}_Z(\varphi)\geq2$ and $f$ is monic with at least two distinct roots. Then $A$ is never isomorphic to a surface of the form $\frac{K[X,Y,Z]}{(X^nY-Q(X,Z))}$ where $n\geq 2$ and $Q$ is monic in $Z$ with ${\rm deg}_Z(Q)\geq 2$.  
          \end{cor}
          
      \indent
      	The following examples show that when ${\rm char}K>0$, the hypothesis ``$n\geq 2$ and $g$ has no root of multiplicity $n$" cannot be dropped from Corollary \ref{c6}. 
      	\begin{ex}
      		\rm
      		Let ${\rm char}K$=2 and $A=\frac{K[X Y,Z]}{(X(X+1)Y-Z^2)}$. Here we have $n=1$. we denote $x, y, z$ to be images of $X, Y, Z$ in $A$ respectively. Define a $k$-algebra homomorphism $T:A \longrightarrow A$ such that $T(x)= x+1$, $T(y)=y$ , $T(z)=z$. Then it is clear that $T$ is a automorphism of $A$.
      	\end{ex} 
      	\begin{ex}
      		\rm
      		Let ${\rm char}K=2$ and $B=\frac{K[X,Y,Z]}{(X^2(X+1)^2Y-Z^2)}$. Here $n=2$ but $g$ has root of multiplicity 2. Similarly here also we have a $k$-algebra automorphism $T$ of $B$ defined by $T(x)= x+1$, $T(y)=y$ , $T(z)=z$ 
      		where  $x,y,z$ denote the images of $X,Y,Z$ in $B$ respectively.
      	\end{ex}

			\section{Stable isomorphism property of Danielewski surfaces}
			\indent
			In this section using method used in \cite{ng_sen_d.danielewski} and \cite{pg_ng_gen_danielewski} we show that a subfamily of General Danielewski surfaces provides counterexamples to the cancellation problem Q\ref{q1}.
			Throughout this section 
			$$A=\frac{K[X,Y,Z]}{(f(X)Y-\varphi{(X,Z)})}$$
			 where $f$ has at least one double root, $\varphi(X,Z)=\sum C_{ij}X^iZ^j$ is monic in $Z$ such that ${\rm deg}_Z(\varphi(X,Z))\geq 2$. Suppose that  $$(\varphi(X,Z),\varphi_Z(X,Z))=K[X,Z].$$  Without loss of generality we assume that the double root of $f$ is zero i.e $f=X^ng(X)$ where $n\geq2$. Let $$h(X)=X^{n-1}g(X)$$ and  $$B=\frac{K[X,Y,Z]}{(h(X)Y-\varphi(X,Z))}.$$ Then we have the following theorem:
			\begin{thm}\label{t7.1}
				\smallskip
				$A^{[1]}\cong B^{[1]}$.  
			\end{thm}
			\begin{proof}
				\smallskip
				Let $x,y,z$ be the images of $X,Y,Z$ in $A$ respectively.
				We have an exponential map $\phi : A\longrightarrow A[U]$ given by
				\begin{alignat*}{3}
				   &\phi(x)=x\\
					&\phi(z)=z+f(x)U\\
					&\phi(y)=\frac{\varphi(x,z+f(x)U)}{f(x)} =y+ U\alpha(x,z,U) \text{ for some $\alpha \in K^{[3]}$}.
			    \end{alignat*}

				Let $R=A[v]\cong A^{[1]} $. Then we can extend $\phi$ to an exponential map $$\phi:A[v] \longrightarrow A[v][U]$$ by defining $\phi{(v)}=v-xU$. 
				Let $\theta=hv+z$. Then $\theta \in R^{\phi}$.
				Also note that
				\begin{equation}
					\begin{split}
						\varphi(x,\theta) & =\varphi(x,hv+z)
						\\&= \sum C_{ij}x^i(hv+z)^j
						\\& =\varphi(x,z)+ h(v\varphi_z(x,z)+xr)  \text{[as higher powers of $h$ are divisible by $xh$]}
					\end{split}
				\end{equation}
				\newline
				for some $r\in R$. Thus
				
				\begin{equation}
					\label{}
					\varphi(x,\theta)=hs.
				\end{equation}
				where $s=xy+v\varphi_z(x,z)+xr$.
				
				\smallskip
				Now as $\varphi(x,\theta)\in R^{\varphi}$ and $R^{\varphi}$ is inert we get that $s\in R^{\varphi} $.
				Now from the given condition we get that there exist $a(X,Z),b(X,Z)$ such that 
				\begin{equation}
					a(X,Z)\varphi_Z(X,Z)+b(X,Z)\varphi(X,Z)=1
				\end{equation}
				
				Note that $a(x,\theta)-a(x,z)\in xR$ (as $\theta-z=hv=x^{n-1 }g(x)v$).
				\begin{equation}
					\begin{split}
						v-sa(x,\theta) & = v-a(x,\theta)[xy+v\varphi_z(x,z)+xr] 
						\\ & =v(1-\varphi_z(x,z)a(x,z))+xt     \text{     for some $t\in R$}
						\\ & =v(b(x,z)\varphi(x,z))+xt
						\\ & =vf(x)yb(x,z)+xt
						\\ & =x(vhyb(x,z)+t)
					\end{split}
				\end{equation}
				
				\smallskip
				Let $w=\frac{v-sa(x,\theta)}{x}$. Then by (11), $w\in R$. Then we get that $\phi(w)=w-U$. Thus by Lemma \ref*{l2.4}(iii), we get 
				\begin{equation}
					R=R^{\varphi}[w]=(R^{\varphi})^{[1]}. 
				\end{equation}
				
				\smallskip
				Let $E=K[x,\theta,s]$. Now we prove that $B\cong E$. Define a $K$-algebra homomorphism $\psi: K[X,Y,Z]\longrightarrow E$ by $\psi(X)=x,\psi(Y)=s,\psi(Z)=\theta$. Clearly this map is surjective. From (9), we get a surjective map $\overline{\psi}$ from $B$ to $E$. Also $2\leq $trdeg$_KE$ $\leq$ trdeg$_K(R^{\varphi})=2$, we get that $\text{dim}(E)=2$. Thus we get that $\overline{\psi}$ is an isomorphism.
				\newline
				
				Now we prove that $E=R^{\varphi}$. Clearly we have $E\subset R^{\varphi}$. Note that
				\begin{equation}
					\begin{split}
						R[x^{-1}] & =K[x^{\pm1},y,z][v]
						\\&= K[x^{\pm1},s,\theta][v]
						\\&= K[x^{\pm1},s,\theta][w]
						\\&=E[x^{-1}][w]
						\\& =R^{\varphi}[x^{-1}][w]  
					\end{split}   
				\end{equation} 
				
				\smallskip
				Hence we get that $E[x^{-1}]=R^{\varphi}[x^{-1}]$. 
				
				\smallskip
				Thus by Lemma \ref{2.11} it is enough to show that $xR^{\varphi}\cap E=xE$. As $xR\cap R^{\varphi}=xR^{\varphi} $, it is enough to show that $xR\cap E=xE$. That is to show that the kernel of the map $\pi: E\longrightarrow \frac{R}{xR}$ is $xE$. 
				
				\smallskip
				Note that from the isomorphism $\overline{\psi}$ defined above we get that 
				\begin{equation}
					\frac{E}{xE}\cong \left(\frac{K[Z_1]}{\varphi(0,Z_1)}\right)[Y_1] 
				\end{equation} 
				
				Also 
				\begin{equation}
					\frac{R}{xR}\cong\left(\frac{K[Z]}{\varphi(0,Z)}\right)[Y,v]
				\end{equation}
				
				\smallskip
				Now $\pi(\theta)=\overline{Z}$ and $\pi(s)=v\varphi_Z(0,\overline{Z})$. As $\varphi_Z(0,\overline{Z})$ is a unit in $\frac{R}{xR}$ by equation (10) we get that 
				\begin{equation}
					\pi(E)=\frac{K[Z]}{\varphi(0,Z)}[v]\cong \left(\frac{K[Z]}{\varphi(0,Z)}\right)^{[1]}
				\end{equation}
				
				From the equations (14), (15) and (16) above we see that $\pi$ induces a map 
				
				$$\overline{\pi}:\frac{E}{xE}\cong \left(\frac{K[Z_1]}{\varphi(0,Z_1)}\right)[Y_1]\longrightarrow \frac{R}{xR}\cong\frac{K[Z]}{\varphi(0,Z)}[Y,v] $$ such that $\overline{\pi}(\overline{Z_1})=\overline {Z}$ and $\overline{\pi}(\overline{Y_1})=v$ which is clearly injective.
				
				\smallskip
				Thus we get that ker($\pi$)=$xE$. And hence from (12) it follows that $R\cong E^{[1]}$ i.e $A^{[1]}\cong B^{[1]}$.  
			\end{proof}
			
			\smallskip
		Now using Corollary \ref{c6.3} and the following theorem, we get a new family of pairwise non-isomorphic surfaces which are stably isomorphic. And thus we get a new family of counter examples to the cancellation problem. 
		\begin{thm}\label{th5.2}
			\smallskip
		Let $g(X)\in K[X]$ be any polynomial with no double root and $\varphi(X,Z)\in K[X,Z]$ be monic in $Z$ with {\rm deg}$_Z(\varphi) \geq2$ such that $(\varphi(X,Z),\varphi_Z(X,Z))=K[X,Z]$. Consider the set \[\Sigma=\left\{A_n\middle|A_n=\frac{K[X,Y,Z]}{(X^ng(X)Y-\varphi{(X,Z)})},n\geq2 \right\}.\] Then $\Sigma$ provides a family of pairwise non-isomorphic surfaces that are stably isomorphic.
		\end{thm}   
			\begin{proof}
				Follows from Theorems \ref{t6.1} and \ref{t7.1}.
			\end{proof}
	
			\appendix
    \section{Appendix}
    		
    	\indent 
    	
    	 In literature, geometric definitions of Danielewski surfaces, accommodating certain subfamily of the surfaces $X^nY-\phi(X,Z)$, play a central role in studying these surfaces geometrically. One such geometric definition has been provided in \cite{jauslin_poloni}.
    	In this section we extend the definition in \cite{jauslin_poloni} in a natural way and we show that our algebraically defined surfaces under consideration satisfy this geometric definition under some conditions.
    	\begin{Defn}
    		\medskip
    		A General Danielewski surface over an algebraically closed field $K$ is a smooth affine algebraic variety S equipped with an $\mathbb{A}^1$-fibration $\pi: S\longrightarrow \mathbb{A}^1$ for which there exist a finite set $\Omega\subset \mathbb{A}^1$ and a fixed positive integer $d\geq2$ such that exceptional fibers at each point of $\Omega$ is a disjoint union of $d$ many lines and the general fibers at any other point is a line.
    	\end{Defn} 
    	\smallskip
    	We now show that under some restriction General Danielewski surfaces defined by a polynomial of the form $f(X)Y-\varphi(Z,X)$ satisfies the above definition of Danielewski surfaces. 
    	\begin{lem}
    		\medskip
    		Let $K$ be an algebraically closed field. Suppose $A=\frac{K[X,Y,Z]}{(f(X)Y-\varphi(X,Z) )} $ where $f=X^r+a_{r-1}X^{r-1}+...+a_1X+a_0$ is a monic polynomial in $X$ and $\varphi(X,Z)=Z^d+c_{d-1}(X)Z^{d-1}+...c_1(X)Z^1+c_0(X)$ such that $d\geq2,r\geq2$. Further suppose that for every root $\lambda$ of $f(X)$ the polynomial $\varphi(Z,\lambda) $ has distinct roots, then the inclusion map $k[x]\longrightarrow A$ is an $\mathbb{A}^1$-fibration for which the corresponding map between the associated affine varieties satisfies the above definition.  
    	\end{lem}
    	\begin{proof}
    		\smallskip
    		Let $S$ be the affine variety having the coordinate ring $A$. From the assumption in the lemma, using the determinant condition it can be easily checked that $S$ is a smooth surface. Also, we note that $A$ is a finitely generated flat algebra over $K[x]$. Let $\alpha=f(\lambda) \neq 0$, in that  case 
    		
    		$$A\otimes_{K[x]}\left(\frac{K[x-\lambda]}{(x-\lambda)}\right)_{(x-\lambda)}=\frac{K[Y,Z]}{(\alpha Y-\varphi(Z,\lambda))}\cong K^{[1]}.$$
    		\indent
    		Now suppose $f(\lambda)=0$ then from the assumption we have $\varphi(Z,\lambda)$ has $d$ many distinct roots say $\mu_1,\mu_2,...,\mu_d$. In this case we have 
    		$$A\otimes_{k[x]}\left(\frac{K[x-\lambda]}{(x-\lambda)}\right)_{(x-\lambda)}\cong \frac{K[Y,Z]}{(\varphi(Z,\lambda))}=\frac{K[Y][Z]}{(Z-\mu_1)}\times \frac{K[Y][Z]}{(Z-\mu_2)}\times...\times\frac{K[Y][Z]}{(Z-\mu_d)} $$
    		\indent
    		This proves that the induced map $i^*: S\longrightarrow \mathbb{A}^1$ satisfies the requirements of an $\mathbb{A}^1$-fibration so that $S$ becomes a Danielewski surface.
    	\end{proof}
    	 It is worth noting that in the geometric definition of Danielewski surface, the surface has to be smooth and hence normal. However if an arbitrary surface is defined by the polynomial $f(X)Y-\varphi(Z,X)$ without the assumptions as in the above lemma then the surface need not be smooth. For example consider the surface $A=\frac{K[X,Y,Z]}{(X^2Y-Z^2 )}$, this surface is not even normal, hence it can not be smooth.
    	
	\bigskip
	\noindent
	{\bf Acknowledgements}
   I thank Prof. Amartya Kumar Dutta, and Prof. Neena Gupta for going through the draft and for their valuable suggestions that improved the readability of this article substantially. My sincere thanks to Parnashree Ghosh for introducing me to the work of A. C. Bianchi and M. O. Veloso \cite{veloso} and for suggesting to investigate their work over arbitrary fields. I thank Prof. Angelo Calil Bianchi for one helpful mail communication between us in 2024.

		\end{document}